\documentclass[reqno]{amsart} \usepackage{graphicx,
  amsmath,amsfonts,amsthm,amssymb, paralist}
\newtheorem{theorem}{Theorem}[section]

\newtheorem{lemma}[theorem]{Lemma}

\theoremstyle{definition} 
\theoremstyle{remark} 
\numberwithin{equation}{section}

\newcommand{\g}{\geqslant} 
  \newcommand{\RR}{\mathbb{R}}
 \newcommand{\CC}{\mathbb{C}}
 
\newcommand{\sph}{\mathbb{S}} 
\newcommand{\p}{\partial} \newcommand{\q}{\varphi}
\newcommand{\les}{\leqslant}

\newcommand{\mc}[1]{\mathcal{#1}} \newcommand{\mb}[1]{\mathbf{#1}}
 \newcommand{\lr}[1]{ \langle #1
  \rangle}

\renewcommand{\Re}{\mathrm{Re}}

\DeclareSymbolFont{bbold}{U}{bbold}{m}{n}
\DeclareSymbolFontAlphabet{\mathbbold}{bbold}

\setdefaultenum{(i)}{(a)}{1}{A}

\begin{document}

\title[On the Majorana condition]{On the Majorana condition for nonlinear Dirac systems}%
\author[T.~Candy]{Timothy Candy}%
\address[T.~Candy]{Universit\"at Bielefeld, Fakult\"at f\"ur
  Mathematik, Postfach 100131, 33501 Bielefeld, Germany}
\email{tcandy@math.uni-bielefeld.de}%
\author[S.~Herr]{Sebastian Herr}%
\address[S.~Herr]{Universit\"at Bielefeld, Fakult\"at f\"ur
  Mathematik, Postfach 100131, 33501 Bielefeld, Germany}
\email{herr@math.uni-bielefeld.de} \thanks{Financial support by the
  German Research Foundation through the CRC 1283 ``Taming uncertainty and profiting from
  randomness and low regularity in analysis, stochastics and their
  applications'' is acknowledged.}%

\begin{abstract}
  For arbitrarily large initial
  data in an open set defined by an approximate Majorana
  condition, global existence and scattering results for
  solutions to the Dirac equation with Soler-type nonlinearity and the Dirac-Klein-Gordon system in
  critical spaces in spatial dimension three are established.
\end{abstract}
\maketitle

\section{Introduction}\label{sec:intro}

Let $m,M\geq 0$. Using the summation convention with respect to
$\mu=0,\ldots,3$, the cubic Dirac equation (Soler model) for a spinor
$\psi: \RR^{1+3} \rightarrow \CC^4$ is given by
\begin{equation}\label{eqn-cubic dirac}
  - i \gamma^\mu \p_\mu \psi + M \psi = (\overline{\psi} \psi) \psi.
\end{equation}
Here, $x^0 = t$, $\p_0 = \p_t$, and $\overline{\psi} = \psi^\dagger \gamma^0$ is the Dirac adjoint, where
$\psi^\dagger$ denotes the complex conjugate transpose of
the spinor $\psi$, and the matrices  $\gamma^\mu \in \CC^{4\times 4}$
are the standard Dirac matrices, see \cite{Candy2016}.
Writing $\Box=\p_t^2-\Delta$, the Dirac-Klein-Gordon system is
\begin{equation} \label{eqn-DKG}
  \begin{split}
    - i \gamma^\mu \p_\mu \psi + M \psi &= \phi \psi, \\
    \Box \phi + m^2 \phi &= \overline{\psi} \psi,
  \end{split}
\end{equation}
where $\phi: \RR^{1+3} \rightarrow \RR$ is a scalar field.  These
equations \eqref{eqn-cubic dirac} and \eqref{eqn-DKG} arise as in
relativistic quantum mechanics as toy models for interactions of
elementary particles, see e.g. \cite{Bjorken1964,Thaller}.

In previous work, we have addressed the initial value problems for the
above equations for small initial data of low regularity. Concerning
the cubic Dirac equation, we have obtained small data global
well-posedness and scattering in the massive case $M>0$
\cite{Bejenaru2014a,Bejenaru2015b} as well as the massless case $M=0$
\cite{Bournaveas2015}.  For the massive Dirac-Klein-Gordon system, we have
obtained small data global well-posedness in the  non-resonant
regime for initial data of subcritical regularity \cite{Bejenaru2015}
and both in the resonant and the non-resonant regime in the critical
space with additional angular regularity \cite{Candy2016}. Concerning
a more complete account on earlier work on the low regularity
well-posedness problem, we refer to the references therein. The purpose
of the current article to gain insight into the asymptotic behaviour
of an open set of large data solutions to \eqref{eqn-cubic dirac} and \eqref{eqn-DKG}.

In \cite{Chadam1974} Chadam and Glassey considered the equations
\eqref{eqn-cubic dirac} and \eqref{eqn-DKG} under the assumption that
the initial data was of the form
\begin{equation}\label{eqn-chadam glassey struct}
  \psi(0) = (f,g,-g^*,f^*)^t\end{equation}
where, given a complex scalar (or vector) $z \in \CC^n$,
we let $z^*$ denote the complex conjugate, and $f, g: \RR^3 \rightarrow \CC$. This condition
\eqref{eqn-chadam glassey struct} is equivalent to
\begin{equation}\label{eq:lm} \psi(0) + z \gamma^2
  \psi^*(0)=0
\end{equation}
with $z=-i$, see \cite{Ozawa2004}. A computation shows that the
condition \eqref{eqn-chadam glassey struct} is conserved under the
evolution of \eqref{eqn-cubic dirac} and \eqref{eqn-DKG}, and
moreover, that if $\psi$ is of the form \eqref{eqn-chadam glassey
  struct} then $\overline{\psi} \psi = 0$.  Consequently, under the
assumption \eqref{eqn-chadam glassey struct}, the cubic Dirac equation
\eqref{eqn-cubic dirac} and the Dirac-Klein-Gordon system
\eqref{eqn-chadam glassey struct} reduce to equations which are linear
in $\psi$. In particular, the argument of Chadam-Glassey gives
scattering and global well-posedness for \eqref{eqn-cubic dirac} and
\eqref{eqn-DKG} for a class of large data \cite{Chadam1974}.  The
structural condition \eqref{eqn-chadam glassey struct} considered by
Chadam and Glassey was introduced in the physics literature long
before by Majorana \cite{Majorana1937} to describe fermions which are
their own anti-particles, see \cite{elliott2015} for an overview.

Our main Theorems \ref{thm-main cubic
  dirac} and \ref{thm-main dkg} below pertain to solutions emanating
from initial data which approximately satisfy the algebraic condition
\eqref{eq:lm} with $|z|=1$. For the results concerning the cubic Dirac equation
\eqref{eqn-cubic dirac}, we rely on the estimates obtain in
\cite{Bejenaru2014a,Bejenaru2015b,Bournaveas2015}. On the other hand, in the case of the Dirac-Klein-Gordon system \eqref{eqn-DKG}, we require more refined estimates than those used in \cite{Candy2016} to obtain the current sharpest small data global theory. The reason is that we have to deal with a large potential in
the Dirac equation, which essentially is a free Klein-Gordon
wave. Instead, we use refined estimates obtained in \cite{Candy2017a} which give a small power of a space-time $L^4_{t,x}$ norm on the righthand side. 

The main result regarding the cubic Dirac equation is the following.

\begin{theorem}\label{thm-main cubic dirac}
  Let $z \in \CC$, $|z|=1$, and $M \g 0$. For any $A\g 1$ there exists
  $\epsilon=\epsilon(A)>0$ such that for all initial data satisfying
        \[\| \psi(0) \|_{H^1(\RR^3)} \les A\; \text{ and }\; \| \psi(0) + z \gamma^2 \psi^*(0) \|_{H^1(\RR^3)} \les \epsilon,\]
        the cubic Dirac equation \eqref{eqn-cubic dirac} is globally well-posed and solutions
        scatter to free solutions as
        $t \rightarrow \pm \infty$.
\end{theorem}
To be more precise, we prove Theorem \ref{thm-cubic dirac reduced} on
a reduced system instead, which is equivalent for smooth solutions.
      In Theorem \ref{thm-main cubic dirac} we are forced to take
      $\epsilon$ much smaller than $A^{-1}$.  The regularity assumption in Theorem
      \ref{thm-main cubic dirac} is sharp, in the sense that
      $\dot{H}^1(\RR^3)$ is the scale invariant space. In particular,
      the regularity assumptions match the optimal results known in
      the small data case \cite{Bejenaru2014a, Bournaveas2015}.  The
      importance of Theorem \ref{thm-main cubic dirac} is that we can
      take $A$ to be large, in particular, we obtain scattering
      for an open set of large data with essentially sharp regularity
      assumptions. Under stronger decay and regularity conditions,
      such results have been proven by Bachelot in
      \cite{Bachelot1989}. Very recently, a similar result has been
      derived in \cite{D'Ancona2017} in the presence of a time
      independent potential and for initial data in $H^1(\RR^3)$ with
      additional angular regularity.

      We also have the corresponding version for the
      Dirac-Klein-Gordon system. Let $H^s_\sigma(\RR^3) = (1-\Delta_{\sph^2})^{-\frac{\sigma}{2}}H^s(\RR^3)$ be the
      subspace of the standard Sobolev space $H^s(\RR^3)$ containing functions with $\sigma$ angular derivatives in $H^s(\RR^3)$, equipped with the norm 
        $$ \| f \|_{H^s_\sigma} = \| ( 1 -\Delta_{\sph^2})^{\frac{\sigma}{2}} f \|_{H^s}$$
       see  \cite{Candy2016, Candy2017a} for details. Note that $H^s(\RR^3)=H^s_0(\RR^3)$.

\begin{theorem}\label{thm-main dkg}
  Let $z \in \CC$, $|z|=1$. Suppose that either $s>0=\sigma$ and
  $2M > m > 0 $, or $\sigma>0=s$ and $M,m>0$. For any $A\g 1$, there
  exists $\epsilon=\epsilon (A)>0$, such that if
    \[  \| \psi(0) \|_{H^s_\sigma(\RR^3)}\les A, \quad \| \phi(0)
      \|_{H^{\frac{1}{2}+s}_\sigma(\RR^3)} \les A, \quad \| \p_t \phi(0) \|_{H^{-\frac{1}{2}+s}_\sigma(\RR^3)} \les A, \]
    and
    \[\big\|\psi(0) + z \gamma^2 \psi^*(0) \big\|_{H^s_\sigma (\RR^3)}  \les \epsilon, \]
    then the system \eqref{eqn-DKG} is globally well-posed and solutions
    scatter  to free solutions as $t \rightarrow \pm \infty$.
  \end{theorem}
As for the cubic Dirac equation, we prove Theorem \ref{thm-dkg
  reduced} on a reduced system instead, which is equivalent for smooth solutions.

We obtain an upper bound for $\epsilon$ which is the inverse
exponential of a power of $A$, see
Theorem \ref{thm-dkg reduced} for more details.
  The Chadam-Glassey result in
  \cite{Chadam1974} corresponds to the case $z=i$ and $\epsilon = 0$
  (with additional smoothness assumptions on the data). A result similar to Theorem \ref{thm-main dkg} under strong decay and regularity conditions has been established in \cite{Bachelot1988}.
  Notice that the small data results in \cite{Bejenaru2014a,
    Bournaveas2015, Candy2016} correspond to
  Theorems \ref{thm-main cubic dirac} and \ref{thm-main dkg}, respectively, in the case where $A$ is very small,
  since it clearly implies the condition on  $\psi(0) + z \gamma^2 \psi^*(0)$. Notice that
   $s=0$ is the critical regularity for \eqref{eqn-DKG}.

\subsection*{Organisation of the paper}
In Section \ref{sec:red} we
 perform an initial reduction which decouples the small and
 the large parts of the spinors. In Section \ref{sec:cd} we reformulate
 and prove the main results concerning the Soler model. In Section \ref{sec:dkg} we reformulate
 and prove the main results on the Dirac-Klein-Gordon system.

\section{Initial reductions}\label{sec:red}

Suppose we have data $\psi(0)$ satisfying the assumptions
  of Theorem \ref{thm-main cubic dirac}. One way to proceed would be
  to linearise around the Chadam-Glassey type solutions. Thus
  decomposing
            \[ \psi(0) = \psi_N(0) + \psi_L(0)\]
            where $\| \psi_N(0) \|_{H^1} \les \epsilon$ and
            $\psi_L(0)+ z \gamma^2 \psi_L^*(0)=0$. Let $\psi_L$ denote
            the solution to the linear Dirac equation with data
            $\psi_L(0)$. As mentioned in the introduction, for all times we have
            $\overline{\psi}_L \psi_L = 0$. Consequently, the
            remaining term $\psi_N = \psi - \psi_L$ satisfies the
            equation
    \[  - i \gamma^\mu \p_\mu \psi_N + M \psi_N = \big( \overline{\psi}_L \psi_N + \overline{\psi}_N \psi_L \big) \psi +  \overline{\psi}_N \psi_N \psi.\]
    The last term is small since $\psi_N(0)$ is small. On the other
    hand, it is not at all clear that the first term
    $\big( \overline{\psi}_L \psi_N + \overline{\psi}_N \psi_L \big)
    \psi$
    should be small, since it contains terms of the schematic form
    $\psi_L^2 \psi_N$, and $\psi_L$ can be large. In particular, if we
    wanted to use the linearised equation to prove Theorem
    \ref{thm-main cubic dirac}, we would be forced to absorb these
    terms into the left hand side, which would significantly
    complicate the required multilinear estimates.
    It turns out that there is a better way to
    decompose $\psi$, which avoids this problem. In particular, we can exploit
    the multilinear estimates already contained in
    \cite{Bejenaru2014a, Bournaveas2015}.
    A similar comment applies to
    the proof for the Dirac-Klein-Gordon system, Theorem \ref{thm-main
      dkg}. However, a significant additional difficulty arises in the
    case where the data for $\phi$ is large.

 We start with the following observation, see \cite{Majorana1937,Lochak1985,Chadam1974}, we follow \cite{Ozawa2004}.

\begin{lemma}\label{lem-chadam glassey}
  Assume that $\psi$ is a classical solution of
	\[ - i \gamma^\mu \p_\mu \psi + M \psi = V \psi \]
        for some real-valued, scalar, and locally integrable function
        $V:\RR^{1+3} \rightarrow \RR$. Then for any $z \in \CC$ we
        have
	\[ \| \psi(t) + z \gamma^2 \psi^*(t) \|_{L^2_x} = \| \psi(0) +z \gamma^2 \psi^*(0) \|_{L^2_x}.\]
      \end{lemma}
      \begin{proof}
        A computation shows that
        $ \gamma^\mu \gamma^2 = - \gamma^2 (\gamma^\mu)^*$ which
        implies that
	\[ - i\gamma^\mu \p_\mu \big( \psi + z \gamma^2 \psi^*\big) =  - i \gamma^\mu \p_\mu \psi  + z \gamma^2 \big(  - i \gamma^\mu \p_\mu \psi)^*= - M ( \psi +  z \gamma^2 \psi^*) + V( \psi + z \gamma^2 \psi^*).\]
        Result now follows by multiplying by
        $ i (\psi + z \gamma^2 \psi^*)^\dagger \gamma^0$, taking the
        real part, and then integrating over $\RR^3$.
      \end{proof}

      We can now rewrite the cubic Dirac equation \eqref{eqn-cubic
        dirac}. Let $\varphi, \chi: \RR^{1+3} \rightarrow \CC^4$ be smooth enough and
      solve
      \begin{equation}\label{eqn-cubic dirac II}
        \begin{split}
          - i \gamma^\mu \p_\mu \varphi + M \varphi &= \big( \overline{\varphi} \chi  + \overline{\chi} \varphi\big) \varphi \\
          - i \gamma^\mu \p_\mu \chi + M \chi &= \big(
          \overline{\varphi} \chi + \overline{\chi} \varphi\big) \chi
        \end{split}
      \end{equation}
      with data
    \[ \varphi(0) = \frac{1}{2} \big( \psi(0) + z \gamma^2 \psi^*(0) \big), \qquad \chi(0) = \frac{1}{2} \big( \psi(0) - z \gamma^2 \psi^*(0)\big).\]
    Then a computation using Lemma \ref{lem-chadam glassey} implies
    that for all $t \in \RR$ and $|z|=1$ we have
    \[ \varphi(t) + z \gamma^2 \varphi^*(t)=0, \qquad \chi(t)  -  z \gamma^2 \chi^*(t)=0\]
    and moreover that
    $\overline{\varphi}\varphi = \overline{\chi} \chi = 0$.
    Consequently, if we let $\psi = \varphi + \chi$, we obtain a
    solution to the cubic Dirac equation (\ref{eqn-cubic
      dirac}). Similarly, in the case of the Dirac-Klein-Gordon system
    \eqref{thm-main dkg}, let
    $\varphi, \chi: \RR^{1+3} \rightarrow \CC^4$ and
    $\phi: \RR^{1+3} \rightarrow \RR$ be smooth enough and solve
    \begin{equation}\label{eqn-dkg-sec}
      \begin{split}
        - i \gamma^\mu \p_\mu \varphi + M \varphi &= \phi \varphi \\
        - i \gamma^\mu \p_\mu \chi + M \chi &= \phi \chi \\
        \Box \phi + m^2 \phi &= \overline{\varphi} \chi +
        \overline{\chi} \varphi
      \end{split}
    \end{equation}
    with data
\[\varphi(0) = \frac{1}{2} \big( \psi(0) + z \gamma^2 \psi^*(0) \big), \qquad \chi(0) = \frac{1}{2} \big( \psi(0) - z \gamma^2 \psi^*(0)\big)
.\]
 As in the case of the cubic Dirac equation, an application of
        Lemma \ref{lem-chadam glassey} implies that
\[ \varphi(t) + z \gamma^2 \varphi^*(t)=0, \qquad \chi(t) -  z \gamma^2 \chi^*(t)=0\]
and hence provided $|z|=1$ we have
$\overline{\varphi}\varphi = \overline{\chi} \chi = 0$. Consequently,
letting $\psi = \varphi + \chi$ we get a solution to \eqref{eqn-DKG}.
For technical reasons, we prefer to work with a first order system. Defining $\phi_+=\phi+i\lr{\nabla}^{-1}\partial_t \phi$, as $\phi$ is real-valued, we obtain
\begin{equation}\label{eqn-dkg II}
      \begin{split}
        - i \gamma^\mu \p_\mu \varphi + M \varphi &= \Re(\phi_+) \varphi \\
        - i \gamma^\mu \p_\mu \chi + M \chi &=  \Re(\phi_+) \chi \\
        -i\partial_t \phi_++\lr{\nabla}_m  \phi_+ &=\lr{\nabla}_{m}^{-1}\big( \overline{\varphi} \chi +
        \overline{\chi} \varphi\big)
      \end{split}
    \end{equation}
with data
\begin{align*}&\varphi(0) = \frac{1}{2} \big( \psi(0) + z \gamma^2 \psi^*(0) \big), \quad \chi(0) = \frac{1}{2} \big( \psi(0) - z \gamma^2 \psi^*(0)\big),\\
\text{and }  &\phi_+(0)=\phi(0)+i\lr{\nabla}^{-1}\partial_t \phi(0).
\end{align*}
Conversely, from $\phi_+$ we can recover $\phi$ by taking the real part of $\phi_+$.

\section{Cubic Dirac equation}\label{sec:cd}
We begin by introducing some notation. Let $\Pi_\pm$ be the projection
\[
\Pi_\pm = \tfrac{1}{2} \big( I \pm \lr{\nabla}_M^{-1} (   - i \gamma^0 \gamma^j \p_j + M \gamma^0) \big),
\]
let $\mc{U}^{\pm}_m(t)=e^{\mp i t \lr{\nabla}_m}$ be the propagator for the
homogeneous half-wave equation, let \[\mc{U}_{M}(t)=\mc{U}^+_M(t)\Pi_+ + \mc{U}^{-}_M(t)\Pi_-\] be the propagator for the homogeneous Dirac equation, and let
\begin{align*}
\mc{I}^{\pm,m}_{t_0}(F)(t)=&i\int_{t_0}^t \mc{U}^{\pm}_m(t-t_0-t') F(t')dt',\\
\mc{I}^{M}_{t_0}(G)(t)=&i\int_{t_0}^t \mc{U}_{M}(t-t_0-t') \gamma^0 G(t')dt'.
\end{align*}
be the corresponding Duhamel integrals.

The previous section implies that for smooth solutions \eqref{eqn-cubic dirac} and  \eqref{eqn-cubic dirac II} are equivalent, so that we focus on proving the following.

\begin{theorem}\label{thm-cubic dirac reduced}
  Let $z \in \CC$, $|z|=1$, and $M\g 0$.  There exists $c\in (0,1)$,
  such that for any $A>0$ and
  $\epsilon\les cA^{-1}$, if the initial data satisfy
        \[\| \varphi(0) \|_{H^1} \les \epsilon, \qquad  \| \chi(0) \|_{H^1} \les A,\]
        then \eqref{eqn-cubic dirac II} is globally well-posed and the solutions
         scatter in $H^1(\RR^3)$ to free solutions as
        $t \rightarrow \pm \infty$,
    i.e. there exist $\varphi_{\pm\infty}\in H^1(\RR^3)$ and
    $\chi_{\pm\infty}\in H^1(\RR^3)$, such that
\[\lim_{t \to \pm\infty}\|\varphi(t)-\mc{U}_{M}(t)
\varphi_{\pm\infty}\|_{H^1}=0 \text{ and } \lim_{t \to \pm\infty}\|\chi(t)-\mc{U}_{M}(t)
\chi_{\pm\infty}\|_{H^1}=0.\]
      \end{theorem}

\begin{proof}
  Let $X \subset C(\RR, H^1(\RR^3))$ be the Banach space constructed
  in \cite{Bejenaru2014a} in the massive case ($M>0$) and in \cite{Bournaveas2015} in the massless case ($M=0$). Further, let $\|\cdot\|_X$ denote the norm obtained by multiplying by the norms from \cite{Bejenaru2014a,Bournaveas2015} by a small enough constant, such that for all solutions $\varphi\in X$ to the inhomogeneous Dirac equation
    \[ - i \gamma^\mu \p_\mu \varphi + M \varphi = \big(\overline{\varphi^{(1)}} \varphi^{(2)} \big) \varphi^{(3)}\]
the bound
  \begin{equation}\label{eqn-thm cubic dirac-X bound} \| \varphi \|_X
    \les \| \varphi(0) \|_{H^1(\RR^3)} + C \| \varphi^{(1)} \|_X \|
    \varphi^{(2)} \|_X \| \varphi^{(3)} \|_X \end{equation}
  holds. Consider the set
	\[\mc{X}= \big\{ (\varphi, \chi) \in X \times X \, \big| \,  \| \varphi \|_X \les 2 \| \varphi(0) \|_{H^1}, \| \chi \|_X \les 2 \| \chi(0) \|_{H^1}\big\}\]
and, for $A,\epsilon>0$, the  norm \[ \|(\varphi, \chi)\|_\mc{X}=\epsilon^{-1}\| \varphi \|_X+A^{-1}\| \chi \|_X. \]
$\mc{X}$ is a complete metric space.        Let $\mc{T} = (T_1, T_2)$ denote the standard (inhomogeneous)
        solution map for \eqref{eqn-cubic dirac II} constructed from
        the Duhamel formula. The bound \eqref{eqn-thm cubic dirac-X
          bound} together with the assumption on the initial data
        show that if $(\varphi, \chi) \in \mc{X}$ then
\begin{align*}
\| T_1(\varphi,\chi)\|_X \les{}& \| \varphi(0) \|_{H^1} + 2C \| \varphi\|_{X}^2 \| \chi\|_{X} \les \| \varphi(0) \|_{H^1} + 2^4C \| \varphi(0)\|_{H^1}^2 \| \chi(0)\|_{H^1}\\
\les{}& ( 1 + 2^4C A \epsilon) \| \varphi(0) \|_{H^1},
\end{align*}
and similarly
	\[\| T_2(\varphi, \chi)\|_X \les \| \chi(0) \|_{H^1} + 2C \| \chi\|_{X}^2 \| \varphi\|_{X} \les ( 1 + 2^4 C A \epsilon) \| \chi(0) \|_{H^1}.\]
        Consequently, provided that $\epsilon \leq (2^4 C A)^{-1}$, we see
        that $\mc{T}: \mc{X} \rightarrow \mc{X}$. Next, we verify that $\mc{T}$ is a contraction.
        For $(\varphi_1,\chi_1),(\varphi_2,\chi_2)\in \mc{X} $ another application of \eqref{eqn-thm cubic
          dirac-X bound} gives
   \[ \| T_1(\q_{1},\chi_1) - T_1(\q_{2},\chi_2) \|_X \les  2^4 C A \epsilon \| \q_1 - \q_{2} \|_X + 2^3 C \epsilon^2 \| \chi_1 - \chi_{2}\|_X, \]
        and similarly
	\[ \| T_2(\q_{1},\chi_{1}) - T_2(\q_{2},\chi_{2}) \|_X \les 2^4 C A \epsilon \| \chi_1 - \chi_{2} \|_X + 2^3 C A^2 \| \q_1 - \q_{2}\|_X.\]
This implies
\[\| \mc{T}(\q_{1},\chi_1) - \mc{T}(\q_{2},\chi_2) \|_{\mc{X}} \les 2^6 CA \epsilon \| (\q_{1},\chi_1) - (\q_{2},\chi_2) \|_{\mc{X}}.\]
        Therefore, choosing $\epsilon \leq (2^7 C A)^{-1}$, the map $\mc{T}: \mc{X}\to \mc{X}$ is a contraction with respect to $\|\cdot\|_{\mc{X}}$, hence it has a unique fixed point in
        $\mc{X}$, and standard arguments show the continuity of the flow map. The scattering claim follows from the finiteness of
        both $ \| \q\|_X$ and $\| \chi\|_X$, because this implies that the pull-backs of $\q$ and $\chi$ along the free evolution, as maps from $\RR$ to $H^1(\RR^3)$, have finite quadratic variation,  see \cite{Bejenaru2014a,Bournaveas2015} for the details.
      \end{proof}

      \section{The Dirac-Klein-Gordon system}\label{sec:dkg}
  Let $P_{\lambda}$ be the standard Littlewood-Paley projections onto dyadic frequencies of size $\lambda$, and take $H_N$ to be the projection onto angular frequencies of size $N$, see \cite[Section 2]{Candy2017a} for precise definitions.  If $s \g 0$ and $\sigma=0$, we define
     \[
\|f\|_{\mb{D}^s_0(I)}=\|\lr{\nabla}^s f\|_{L^4(I\times \RR^3)}.
\]
 On the other hand, for $s \g 0$ and $\sigma>0$, we take
\[
\|f\|_{\mb{D}^s_\sigma(I)}=\Big(\sum_{N \g 1}N^{2\sigma} \|\lr{\nabla}^sH_N f\|^2_{L^4(I\times \RR^3)}\Big)^{\frac12}.
\]

The results in Section \ref{sec:red} imply that for smooth solutions
\eqref{eqn-DKG} and \eqref{eqn-dkg II} are equivalent, so that we focus on proving the following.

\begin{theorem}\label{thm-dkg reduced}
  Let $z \in \CC$, $|z|=1$. Suppose that either $s>0=\sigma$ and
  $2M > m > 0 $, or $\sigma>0=s$ and $M,m>0$. There exist $0<c<1$ and $\gamma>1$,
  such that for any $A\g 1$ and any
  $\epsilon\les c\exp(-A^\gamma)$, if
    \[  \| \varphi(0) \|_{H^s_\sigma(\RR^3)}\les \epsilon, \quad \| \chi(0)
      \|_{H^s_\sigma(\RR^3)}\les A,\quad \| \phi_+(0) \|_{H^{\frac{1}{2}+s}_\sigma(\RR^3)} \les A, \]
   then the system \eqref{eqn-dkg II} is globally well-posed and
    scatters to free solutions as $t \rightarrow \pm \infty$,
    i.e. there exist $\varphi_{\pm\infty}\in H^s_\sigma(\RR^3)$,
    $\chi_{\pm\infty}\in H^s_\sigma(\RR^3)$ and  $\phi_{\pm\infty}\in
    H^{\frac{1}{2}+s}_\sigma(\RR^3)$, such that
\begin{align*}
&\lim_{t \to \pm\infty}\|\varphi(t)-\mc{U}_{M}(t)
\varphi_{\pm\infty}\|_{H^s_\sigma}=0, \; \lim_{t \to \pm\infty}\|\chi(t)-\mc{U}_{M}(t)
\chi_{\pm\infty}\|_{H^s_\sigma}=0,\\
\text{ and }&\lim_{t \to \pm\infty}\|\phi_+(t)-\mc{U}^{+}_m(t)
\phi_{\pm\infty}\|_{H^{s+\frac12}_\sigma}=0.
\end{align*}
\end{theorem}

Before we turn to its proof,  we summarise the results we require from \cite{Candy2017a}.

\begin{lemma}\label{lem:norms}
Let $s,\sigma \in\RR$, and $I$ be any interval of the form
$I=[t_1,t_2)$, $-\infty<t_1<t_2\leq \infty$. There exist Banach
function spaces $\mb{F}^{s, \sigma}_M(I)$ and $\mb{V}^{s, \sigma}_{+, m}(I)$ and
$C_0\g 1$ with the following properties:
\begin{enumerate}
\item\label{it:emb}$C_0^\infty(I\times \RR^3;\CC^4)\subset \mb{F}^{s, \sigma}_M(I)$, $C_0^\infty(I\times \RR^3;\CC)\subset \mb{V}^{s, \sigma}_{+, m}(I)$, and
\[
 \mb{F}^{s, \sigma}_M(I)\hookrightarrow C_b(I;H^s_\sigma(\RR^3;\CC^4)),\qquad
 \mb{V}^{s, \sigma}_{+, m}(I)\hookrightarrow C_b(I;H^s_\sigma(\RR^3;\CC)).
\]
\item\label{it:nested} For $\psi \in \mb{F}^{s, \sigma}_M(I)$,
  $\phi_+\in \mb{V}^{s, \sigma}_{+, m}(I)$, and for any $I'=[s_1,s_2)\subset I$, we have $\psi|_{I'} \in \mb{F}^{s, \sigma}_M(I')$,
  $\phi_+|_{I'}\in \mb{V}^{s, \sigma}_{+, m}(I')$, and
\[
\|\psi|_{I'}\|_{\mb{F}^{s, \sigma}_M(I')}\les C_0\|\psi\|_{\mb{F}^{s, \sigma}_M(I)}, \qquad \|\phi_+|_{I'}\|_{\mb{V}^{s+\frac12, \sigma}_{+, m}(I')}\les C_0\|\phi\|_{\mb{V}^{s+\frac12, \sigma}_{+, m} (I)}.
\]
\item\label{it:energy} For $\psi_0 \in H^s_\sigma(\RR^3;\CC^4)$ and $\phi_{0}\in H^s_\sigma(\RR^3;\CC)$ we have $\mc{U}_{M}(t)\psi_0 \in \mb{F}^{s, \sigma}_M(I)$, $\mc{U}^{+}_m (t)\phi_{0}\in \mb{V}^{s, \sigma}_{+, m}(I)$, and the bounds
\begin{equation}\label{eq:energy}
\|\mc{U}_{M} \psi_0\|_{\mb{F}^{s, \sigma}_M(I)}\les \|\psi_0\|_{H^s_\sigma}, \qquad \|\mc{U}^{+}_m \phi_{0}\|_{\mb{V}^{s, \sigma}_{+, m}(I)}\les \|\phi_{0}\|_{H^s_\sigma}.
\end{equation}
\item\label{it:limits} For $\psi \in \mb{F}^{s, \sigma}_M([t_1,t_2))$ and
  $\phi_+\in \mb{V}^{s, \sigma}_{+, m}([t_1,t_2))$ the limits \[\lim_{t\to t_2}\mc{U}_{M}
  (-t)\psi(t) \in H^s(\RR^3;\CC^4)\text{ and }\lim_{t\to
    t_2}\mc{U}^{+}_m(-t)\phi_+(t) \in H^s(\RR^3;\CC)\]exist.
\item\label{it:l4} For $\phi_+\in \mb{V}^{s+\frac12, \sigma}_{+, m}(I)$ we have the Strichartz-type estimate
\begin{equation}\label{eq:l4}
\|\phi_+\|_{\mb{D}^{s}_\sigma (I)} \les{}C_0 \|\phi_+\|_{\mb{V}^{s+\frac12, \sigma}_{+, m}(I)}.
\end{equation}
\item\label{it:nonl} Suppose that either $s>0=\sigma$ and
  $2M > m > 0 $, or $\sigma>0=s$ and $M,m>0$. There exists $\theta\in (0,1)$, such that for any $ t_0 \in I$ the Duhamel operators
\begin{align*}
\mb{V}^{s+\frac12, \sigma}_{+, m}(I)\times \mb{F}^{s, \sigma}_M(I)\ni (\phi_+ ,\varphi) &\mapsto \mc{I}^{M}_{t_0}(\Re(\phi_+) \varphi)\in \mb{F}^{s, \sigma}_M(I),\\
\mb{F}^{s,\sigma}_M(I)\times \mb{F}^{s, \sigma}_M(I)\ni (\chi,\varphi) &\mapsto \mc{I}^{+,m}_{t_0}(\lr{\nabla}_m^{-1} ( \overline{\chi} \varphi))\in \mb{V}^{s+\frac12, \sigma}_{+, m}(I)
\end{align*}
are well-defined and the following estimates hold:
\begin{align}\|\mc{I}^{M}_{t_0}(\Re(\phi_+) \varphi)\|_{\mb{F}^{s, \sigma}_M(I)}\les{}&C_0 \|\phi_+\|_{\mb{D}^s_\sigma(I)}^{\theta}\|\phi_+\|_{\mb{V}^{s+\frac12, \sigma}_{+, m}(I)}^{1-\theta}\|\varphi\|_{\mb{F}^{s, \sigma}_M(I)},\label{eq:nonl1}\\
\|\mc{I}^{+,m}_{t_0}(\lr{\nabla}_m^{-1}(\overline{\chi} \varphi))\|_{\mb{V}^{s+\frac12, \sigma}_{+, m}(I)}\les{}&C_0 \|\chi\|_{\mb{F}^{s, \sigma}_M(I)}\|\varphi\|_{\mb{F}^{s, \sigma}_M(I)}.\label{eq:nonl2}
\end{align}
\end{enumerate}
\end{lemma}
\begin{proof}
For details see Section 2, Lemma 2.1, and Theorem 3.2 in \cite{Candy2017a}.
\end{proof}

The first step in the proof of Theorem \ref{thm-dkg reduced}, is to prove the following local result.
\begin{theorem}\label{thm-DKG local}
Suppose that either $s>0=\sigma$ and
  $2M > m > 0 $, or $\sigma>0=s$ and $M,m>0$.
There exist $\theta,c\in(0,1)$ and $C>1$, such that for any $A, B\g 1$ and any $0<\alpha\les c A^{-1}$ and $0<\beta\les c B^{\frac{\theta-1}{\theta}}$,
and for any interval
  $I=[t_1,t_2)\subset \RR$ and $t_0\in I$, if we have
    \[ \|  \varphi_0 \|_{H_{\sigma}^s(\RR^3)} \les \alpha, \qquad \| \chi_0 \|_{H_{\sigma}^s(\RR^3)} \les A,\]
    and
    \[ \| \mc{U}^{+}_m(\cdot-t_0)\phi_0\|_{\mb{D}_{\sigma}^{s}(I)} \les \beta ,\qquad\|\phi_0 \|_{H_{\sigma}^{\frac12+s}(\RR^3)} \les B,\]
    then there exists a unique solution $(\varphi,\chi, \phi_+)\in \mb{F}^{s, \sigma}_M(I)\times \mb{F}^{s, \sigma}_M(I)\times \mb{V}^{s+\frac12, \sigma}_{+, m}(I)$ of
    \eqref{eqn-dkg II} on $I \times \RR^3$ with initial condition $(\varphi, \chi, \phi_+)(t_0) = (\varphi_0, \chi_0, \phi_0)$. Moreover the solution depends continuously on the initial data and satisfies the bounds
        \begin{align*}& \sup_{t \in I}\| \varphi(t) \|_{H_{\sigma}^s(\RR^3)} \les{} 2 \| \varphi_0 \|_{H_{\sigma}^s(\RR^3)}, \qquad \sup_{t \in I}\| \chi(t) \|_{H_{\sigma}^s(\RR^3)} \les 2 \| \chi_0 \|_{H_{\sigma}^s(\RR^3)},  \\
&\sup_{t \in I} \| \phi_+(t) -\mc{U}^{+}_m(t-t_0)\phi_0(t_0)\|_{H_{\sigma}^{\frac12+s}(\RR^3)} \les{}  C\| \varphi_0 \|_{H_{\sigma}^s(\RR^3)}\| \chi_0 \|_{H_{\sigma}^s(\RR^3)}.
        \end{align*}
      \end{theorem}
\begin{proof}
For convenience, let $\varphi_L(t)=\mathcal{U}_M(t-t_0)\varphi_0$,
$\chi_L(t)=\mathcal{U}_M(t-t_0)\chi_0$, and
$\phi_{+,L}(t)=\mc{U}^{+}_m(t-t_0)\phi_{0}$.
 Let $C_0\g 1$ and $\theta\in (0,1)$ be as in Lemma \ref{lem:norms}.
Define $S$ as the set of all $(\varphi,\chi,\phi_+)\in \mb{F}^{s, \sigma}_M(I)\times \mb{F}^{s, \sigma}_M(I)\times
  \mb{V}^{s+\frac12, \sigma}_{+, m}(I)$ satisfying
\begin{align*}
\|\varphi -\varphi_L\|_{\mb{F}^{s, \sigma}_M(I)}\les{}&
  \|\varphi_0\|_{H_{\sigma}^s},\qquad \|\chi-\chi_L\|_{\mb{F}^{s, \sigma}_M(I)}\les
  \|\chi_0\|_{H_{\sigma}^s}, \\
  \|\phi_+-\phi_{+,L}\|_{\mb{V}^{s+\frac12, \sigma}_{+, m}(I)}\les{}&
  2^3C_0\|\varphi_0\|_{H_{\sigma}^s}\|\chi_0\|_{H_{\sigma}^s}.
\end{align*}
It is a complete metric space with respect to the norm
\[
\|(\varphi,\chi,\phi_+)\|_S:=\alpha^{-1}\|\varphi\|_{\mb{F}^{s, \sigma}_M(I)}+A^{-1}\|\chi\|_{\mb{F}^{s, \sigma}_M(I)}+\eta^{-1}\|\phi_+\|_{\mb{V}^{s+\frac12, \sigma}_{+, m}(I)},
\]
where $\eta>0$ will be chosen later.
Let \[\mathcal{T}=(T_1,T_2,T_3): \mb{F}^{s, \sigma}_M(I)\times \mb{F}^{s, \sigma}_M(I)\times
  \mb{V}^{s+\frac12, \sigma}_{+, m}(I)\to  \mb{F}^{s, \sigma}_M(I)\times \mb{F}^{s, \sigma}_M(I)\times
  \mb{V}^{s+\frac12, \sigma}_{+, m}(I)\] be defined as
\[
\mathcal{T}(\varphi,\chi,\phi_+)=\begin{pmatrix}
\mathcal{U}_M(\cdot-t_0)\varphi_0+\mathcal{I}^M_{t_0}(\Re(\phi_+)\varphi)\\
\mathcal{U}_M(\cdot-t_0)\chi_0+\mathcal{I}^M_{t_0}(\Re(\phi_+)\chi)\\
\mc{U}^{+}_m(\cdot-t_0)\phi_{+,0}+\mc{I}^{+,m}_{t_0}(\lr{\nabla}_m^{-1}(\overline{\varphi}\chi+\overline{\chi}\varphi))
\end{pmatrix},
\]
see Lemma \ref{lem:norms}. Fixed points of $\mc{T}$ are solutions of
\eqref{eqn-dkg II} with the given data at time $t_0$. For
$(\varphi,\chi,\phi_+)\in S$ we infer that
\begin{align*}
\|\varphi\|_{\mb{F}^{s, \sigma}_M(I)}\les{}&
\|\varphi-\varphi_L\|_{\mb{F}^{s, \sigma}_M(I)}+\|\varphi_L\|_{\mb{F}^{s, \sigma}_M(I)}\les
2\|\varphi_0\|_{H_{\sigma}^s}\les 2\alpha,\\
\|\chi\|_{\mb{F}^{s, \sigma}_M(I)}\les{}&
\|\chi-\chi_L\|_{\mb{F}^{s, \sigma}_M(I)}+\|\chi_L\|_{\mb{F}^{s, \sigma}_M(I)}\les
2\|\chi_0\|_{H_{\sigma}^s}\les 2A,
\end{align*}
and similarly,
\begin{align*}
\|\phi_{+,L}\|_{\mb{D}_{\sigma}^s(I)}^{\theta}\|\phi_{+,L}\|_{\mb{V}^{s+\frac12, \sigma}_{+, m}(I)}^{1-\theta}\les{}&
\beta^\theta B^{1-\theta}, \\
 \|\phi_+-\phi_{+,L}\|_{\mb{D}_{\sigma}^s(I)}^{\theta}\|\phi_+-\phi_{+,L}\|_{\mb{V}^{s+\frac12, \sigma}_{+, m}(I)}^{1-\theta}
\les{}&  2^3 C_0^{1+\theta}\|\varphi_0\|_{H_{\sigma}^s}\|\chi_0\|_{H_{\sigma}^s}\les
        2^3C_0^2\alpha A.
\end{align*}
If $\alpha\les (2^5C_0^3A)^{-1}$ and $\beta \les (4C_0 B^{1-\theta})^{-\frac{1}{\theta}}$, Lemma \ref{lem:norms} implies
\begin{equation}\label{eq:apriori1}
\|T_1(\varphi,\chi,\phi_+)-\varphi_L\|_{\mb{F}^{s, \sigma}_M(I)}\les{}\big(2C_0 \beta^\theta  B^{1-\theta} +
                                                            2^4C_0^{3}\alpha
                                                            A\big)\|\varphi_0\|_{H_{\sigma}^s}
\les{}\|\varphi_0\|_{H_{\sigma}^s},
\end{equation}
and
\begin{equation}\label{eq:apriori2}
\|T_2(\varphi,\chi,\phi_+)-\chi_L\|_{\mb{F}^{s, \sigma}_M(I)}\les{}\big(2C_0 \beta^\theta  B^{1-\theta} +
                                                            2^4C_0^{3}\alpha
                                                            A\big)\|\chi_0\|_{H_{\sigma}^s}
\les{}\|\chi_0\|_{H_{\sigma}^s},
\end{equation}
as well as
\begin{equation}\label{eq:apriori3}
\|T_3(\varphi,\chi,\phi_+)-\phi_{+,L}\|_{\mb{V}^{s+\frac12, \sigma}_{+, m}(I)}\les{}2^3C_0 \|\varphi_0\|_{H_{\sigma}^s}\|\chi_0\|_{H_{\sigma}^s}.
\end{equation}
We will now show that $\mc{T}:S\to S$ is a contraction, provided that
$\alpha,\beta$ are chosen small enough. Let $(\varphi,\chi,\phi_+)\in
S$ and $(\tilde{\varphi},\tilde{\chi},\tilde{\phi}_+)\in S$. Then, by
Lemma \ref{lem:norms},
\begin{align*}
\|T_1(\varphi,\chi,\phi_+)-T_1
  (\tilde{\varphi},\tilde{\chi},\tilde{\phi}_+)\|_{\mb{F}^{s, \sigma}_M(I)}\les{}&\big(C_0 \beta^\theta  B^{1-\theta} +
                                                            2^3C_0^{3}\alpha
                                                            A\big)\|\varphi-\tilde\varphi\|_{\mb{F}^{s, \sigma}_M(I)}\\
&+2C_0^2\alpha\|\phi_+-\tilde{\phi}_+\|_{\mb{V}^{s+\frac12, \sigma}_{+, m}(I)},
\end{align*}
and
\begin{align*}
\|T_2(\varphi,\chi,\phi_+)-T_2
  (\tilde{\varphi},\tilde{\chi},\tilde{\phi}_+)\|_{\mb{F}^{s, \sigma}_M(I)}\les{}&\big(C_0 \beta^\theta  B^{1-\theta} +
                                                            2^3C_0^{3}\alpha
                                                            A\big)\|\chi-\tilde\chi\|_{\mb{F}^{s, \sigma}_M(I)}\\
&+2C_0^2A\|\phi_+-\tilde{\phi}_+\|_{\mb{V}^{s+\frac12, \sigma}_{+, m}(I)},
\end{align*}
as well as
\begin{align*}
\|T_3(\varphi,\chi,\phi_+)-T_3
  (\tilde{\varphi},\tilde{\chi},\tilde{\phi}_+)\|_{\mb{V}^{s+\frac12, \sigma}_{+, m}(I)}\les{}
&2^2C_0 \alpha  \|\chi-\tilde\chi\|_{\mb{F}^{s, \sigma}_M(I)}
+2^2C_0   A\|\varphi-\tilde\varphi\|_{\mb{F}^{s, \sigma}_M(I)}.
\end{align*}
We obtain
\begin{align*}
\|\mc{T}(\varphi,\chi,\phi_+)-&\mc{T}
  (\tilde{\varphi},\tilde{\chi},\tilde{\phi}_+)\|_S\les{}4C_0^2\eta \eta^{-1}\|\phi_+-\tilde{\phi}_+\|_{\mb{V}^{s+\frac12, \sigma}_{+, m}(I)}\\
+& \big(C_0 \beta^\theta  B^{1-\theta} +
                                                            2^3C_0^{3}\alpha
                                                            A+2^2C_0 A
  \alpha\eta^{-1}\big)\alpha^{-1}\|\varphi-\tilde\varphi\|_{\mb{F}^{s, \sigma}_M(I)}\\
+&\big(C_0 \beta^\theta  B^{1-\theta} +
                                                            2^3C_0^{3}\alpha
                                                            A+2^2C_0 A
  \alpha\eta^{-1}\big)A^{-1}\|\chi-\tilde\chi\|_{\mb{F}^{s, \sigma}_M(I)}.
\end{align*}
By fixing $\eta=(2^4 C_0^2)^{-1}$, and choosing $\alpha\les
(2^{12}C_0^3A)^{-1}$ and $\beta\les
(2^4C_0B^{1-\theta})^{-\frac{1}{\theta}}$, we have verified that
$\mathcal{T}:S\to S$ is a
contraction, hence it has a fixed point
$(\varphi,\chi,\phi_+)\in S$ which is unique in $S$. For later purposes we note that we have
chosen the thresholds for $\alpha$ and $\beta$ small enough such that
the same conclusion holds
if $\alpha$, $A$, and $B$ are doubled.
Similar estimates show that the fixed point depends
continuously on the initial data. Due to \eqref{eq:energy}, the
claimed estimates on the Sobolev norms for
$(\varphi(t),\chi(t),\phi_+(t))$ for $t\in I$ follow from
\eqref{eq:apriori1}, \eqref{eq:apriori2} and \eqref{eq:apriori3}.

Finally, we prove uniqueness. Assume that $(\varphi',\chi',\phi'_+)\in \mb{F}^{s, \sigma}_M(I)\times \mb{F}^{s, \sigma}_M(I)\times
  \mb{V}^{s+\frac12, \sigma}_{+, m}(I)$ is another solution with the same data at
  $t_0$ such that \[t':=\sup\{t\in I\mid
  (\varphi',\chi',\phi'_+)(t)=
  (\varphi,\chi,\phi_+)
  (t)\}<t_2.\]
Then,
\[
\|\varphi'(t')\|_{H_{\sigma}^s}\les 2\alpha
,\qquad \|\chi'(t')\|_{H_{\sigma}^s}\les 2A,\qquad
\|\phi'_+(t')\|_{H_{\sigma}^{\frac12+s}}\les 2B.
\]
Let $\|\phi'_+\|_{\mb{V}^{s+\frac12, \sigma}_{+, m}(I)}\leq R$.
By Lemma \ref{lem:norms} we have
\[\|\phi'_+\|_{\mb{D}_{\sigma}^{s}(I')}\les C_0\|\phi'_+\|_{\mb{V}^{s+\frac12, \sigma}_{+, m}(I')}\les
C_0^2 R\] for any $I'\subseteq I$. For $\varepsilon\in(0, \beta)$ (which
will be specified below),
let $\delta>0$ be small enough
such that $I':=[t',t'+\delta)\subset I$ and
$\|\phi'_+\|_{\mb{D}_{\sigma}^{s}(I')}\les \varepsilon$.
Let $\varphi'_L(t):=\mathcal{U}_M(t-t')\varphi(t')$,
$\chi'_L(t):=\mathcal{U}_M(t-t')\chi(t')$, and
$\phi'_{+,L}(t):=\mc{U}^{+}_m(t-t')\phi_{+}(t')$. Then,
\[
\|\varphi'-\varphi'_L\|_{\mb{F}^{s, \sigma}_M(I')}\les{}C_0 \varepsilon^\theta
R^{1-\theta}\big(\|\varphi'-\varphi'_L\|_{\mb{F}^{s, \sigma}_M(I')}+\|\varphi'_L\|_{\mb{F}^{s, \sigma}_M(I')}\big),\]
so that if we fix some $\varepsilon\les
(2C_0R^{1-\theta})^{-\frac{1}{\theta}}$, we obtain
\[\|\varphi'-\varphi'_L\|_{\mb{F}^{s, \sigma}_M(I')}\les{}\|\varphi(t')\|_{H_{\sigma}^s}.\]
A similar estimate shows
\[\|\chi'-\chi'_L\|_{\mb{F}^{s, \sigma}_M(I')}\les{}\|\chi(t')\|_{H_{\sigma}^s}.\]
Then,
\[
\|\phi'_{+}-\phi'_{+,L}\|_{\mb{V}^{s+\frac12, \sigma}_{+, m}(I')}\les 2^3C_0 \|\varphi(t')\|_{H_{\sigma}^s}\|\chi(t')\|_{H_{\sigma}^s}.
\]
These estimates show that $(\varphi',\chi',\phi'_+)$ is contained in
the set $S$ defined as above, but with the modified initial condition at $t'$
instead of $t_0$
and the interval $I'$ instead of $I$. Also, the estimates with $I$
replaced by $I'$ in the first
part of the proof imply that $(\varphi,\chi,\phi_+)|_{I'}$ is contained in
this version of the set $S$. The uniqueness within $S$ proven
above implies that $(\varphi',\chi',\phi'_+)=(\varphi,\chi,\phi_+)$ in
$I'$, which contradicts the definition of $t'$.
\end{proof}
We can now prove Theorem
  \ref{thm-dkg reduced} as follows.
 By our hypothesis, the initial data at
 time $0$ satisfy
        \[\|\varphi_0\|_{H_{\sigma}^s}\les \epsilon,\; \|\chi_0\|_{H_{\sigma}^s}\les A,
          \;  \| \phi_{0}\|_{H_{\sigma}^{\frac{1}{2}+s}} \les A,\]
and $\epsilon>0$ is chosen small enough, depending on $A$ only (the precise
threshold will be specified below).
Let $\beta^\ast(B)=cB^{\frac{\theta-1}{\theta}}$ and $\alpha^\ast(A)=cA^{-1}$
be the thresholds as in Theorem \ref{thm-DKG local}.
        Then, by the Strichartz estimate from Lemma \ref{lem:norms} \eqref{it:l4}, we have
    \[ \| \mc{U}^{+}_m(t)\phi_{0} \|_{\mb{D}_{\sigma}^s(\RR_+)} \les C_0 A\]
with $C_0\g 1$. By monotone convergence, the function $T \mapsto \| \mc{U}^{+}_m(t)\phi_{0}
\|_{\mb{D}_{\sigma}^s([T_0,T))}$ is continuous in $T$ and converges to zero as
$T\searrow T_0$. Therefore, for $\beta:=\beta^\ast(2A)$, we can choose
 $0=s_0<s_1<\ldots<s_N$ such that
\[
 \| \mc{U}^{+}_m(t)\phi_{0}
\|_{\mb{D}_{\sigma}^s([s_{n-1},s_n))}=\beta/4 \text{ and } \| \mc{U}^{+}_m(t)\phi_{0}
\|_{\mb{D}_{\sigma}^s([s_n,\infty))}\les \beta/4.
\]
With $s_{N+1}=\infty$, define the collection
    of intervals $I_n=[s_{n-1}, s_{n+1})$ for $n=1, ..., N$. Then,
        \[ \beta/4\les \| \mc{U}^+_m(t)\phi_0 \|_{\mb{D}_{\sigma}^s(I_n)} \les
          \beta/2 \]
and, by Minkowski's inequality,
\[
\sum_{n=1}^N \| \mc{U}^+_m(t)\phi_0 \|_{\mb{D}_{\sigma}^s(I_n)}
^4\les 2 (C_0A)^4,
\]
therefore $N\les N_0:=2^6(C_0A)^4\beta^{-4}$.

Now, fix $\epsilon\les
cC^{-1}C_0^{-1} 2^{-2N_0}A^{-1}\beta$.
    We claim that for
    every $ 1\les n \les N$, on $I_n$ we have a unique solution
    $(\varphi^{(n)},\chi^{(n)},\phi_+^{(n)})\in  \mb{F}^{s, \sigma}_M(I_n)\times \mb{F}^{s, \sigma}_M(I_n)\times
  \mb{V}^{s+\frac12, \sigma}_{+, m}(I_n)$ with initial condition
  \begin{align*}
    (\varphi^{(n)},\chi^{(n)},\phi_+^{(n)})(s_{n-1})=&(\varphi^{(n-1)},\chi^{(n-1)},\phi_+^{(n-1)})(s_{n-1})\quad
  (\text{if }2\les n\les N)\\
  (\varphi^{(1)},\chi^{(1)},\phi_+^{(1)})(s_0)=&(\varphi_0,\chi_0,\phi_{0})
  \quad
  (\text{if }n=1)
  \end{align*}
  which satisfies the bounds
    \begin{equation}  \label{eqn-induc assump}\begin{split}
\big\| \mc{U}^{+}_m(\cdot -s_{n-1})\phi_+^{(n-1)}(s_{n-1})
\big\|_{\mb{D}_{\sigma}^s(I_{n})}\les \beta,\\
\|        \varphi^{(n)}(s_n)\|_{H_{\sigma}^s} \les 2^n \epsilon, \qquad \|\chi^{(n)}(s_n) \|_{H_{\sigma}^s} \les 2^n A, \qquad \\
        \| \phi_+^{(n)}(s_n) - \mc{U}^{+}_m(s_n)\phi_{0}\|_{H_{\sigma}^{\frac{1}{2}+s} }
        \les C 2^{2n} \epsilon A,
  \end{split}
    \end{equation}
where $C$ is the constant from Theorem \ref{thm-DKG local}.
    Indeed, for $n=1$ the estimate in the first line follows by
    definition of $I_1$, and the estimates in  the second and third line follow
    from an application of Theorem \ref{thm-DKG local} (with
    $t_0=0$), where we use that $\epsilon \les \alpha^\ast(A)$ and $\beta
    \les \beta^\ast(A)$.
    As an induction hypothesis, let us suppose that  holds
    \eqref{eqn-induc assump} for some $1\les n \les N-1$.

By Lemma \ref{lem:norms}, the induction hypothesis, and the choice of $\epsilon$ we have
    \begin{align*} \big\| \mc{U}^{+}_m(\cdot -s_n)\phi_+^{(n)}(s_n) \big\|_{\mb{D}_{\sigma}^s(I_{n+1})}
      \les{}& \| \mc{U}^{+}_m \phi_{0} \|_{\mb{D}_{\sigma}^s(I_{n+1})}\\
& {}+   \big\|
        \mc{U}^{+}_m\big( \phi_{0} - \mc{U}^{+}_m(-s_n)
        \phi_+^{(n)}(s_n)\big)\big\|_{\mb{D}_{\sigma}^s(I_{n+1})}
\\
\les{}&  \beta/2 +  C_0 \big\|
  \phi_{0} - \mc{U}^{+}_m(-s_n)
  \phi_+^{(n)} (s_n)\big\|_{H_{\sigma}^{\frac{1}{2} +s}}
\\
      \les{}& \beta/2 + C C_02^{2n} \epsilon A \les \beta.
    \end{align*}
From the estimate in the third line of the induction hypothesis and
the smallness condition on $\epsilon$ we
obtain
\[     \| \phi_+^{(n)} (s_n)\|_{H_{\sigma}^{\frac{1}{2}+s} }
        \les  \|\mc{U}^{+}_m(s_n)\phi_{0}\|_{H_{\sigma}^{\frac{1}{2}+s} }+ C 2^{2n}
        \epsilon A\les A + C 2^{2n}
        \epsilon A\les 2A.
\]
Notice that due to our choices we have
$\beta\les \beta^\ast (2A)$ and $2^n\epsilon\les \alpha^\ast (2^n A)$.
Then, as $s_{n+1}\in I_{n+1}$,
    we obtain from Theorem
    \ref{thm-DKG local} (with $t_0=s_n$) that
 \begin{align*}
   \|      \varphi^{(n+1)} (s_{n+1})\|_{H_{\sigma}^s} \les{}& 2 \|\varphi^{(n)} (s_{n})\|_{H_{\sigma}^s}\les
         2^{n+1} \epsilon,\\
\|\chi^{(n+1)} (s_{n+1}) \|_{H_{\sigma}^s} \les{}&
         2\|\chi^{(n)} (s_{n}) \|_{H_{\sigma}^s} \les 2^{n+1} A,
 \end{align*}
 and, using the induction hypothesis again,
\begin{align*}
       & \| \phi_+^{(n+1)}(s_{n+1}) - \mc{U}^{+}_m(s_{n+1})\phi_{0}\|_{H_{\sigma}^{\frac{1}{2}+s} }\\
        \les{}&\| \phi_+^{(n+1)}(s_{n+1}) -
                \mc{U}^{+}_m(s_{n+1}-s_{n})\phi_+^{(n)}(s_{n})\|_{H_{\sigma}^{\frac{1}{2}+s}}
  \\
&+\| \mc{U}^{+}_m(s_{n+1}-s_{n})\phi_+^{(n)}(s_{n}) - \mc{U}^{+}_m(s_{n+1})\phi_{0}\|_{H_{\sigma}^{\frac{1}{2}+s} }\\
 \les{}&C 2^{2n} \epsilon A+C 2^{2n} \epsilon A\les
         C2^{2(n+1)}\epsilon A.
 \end{align*}
The proof of the claim is complete.

By uniqueness, we have constructed a global
solution \[(\varphi,\chi,\phi_+)\in  C_b(\RR_+,H_{\sigma}^s)\times
  C_b(\RR_+,H_{\sigma}^s)\times C_b(\RR_+,H_{\sigma}^{\frac12+s}),\]
and due to $(\varphi,\chi,\phi_+)|_{[s_N,\infty)}\in \mb{F}^{s, \sigma}_M([s_N,\infty))\times \mb{F}^{s, \sigma}_M([s_N,\infty))\times
  \mb{V}^{s+\frac12, \sigma}_{+, m}([s_N,\infty)) $
 it scatters as $t\to \infty$, see Lemma \ref{lem:norms} Part
 \eqref{it:limits}. The claim for $t\to -\infty$ follows by time
 reversibility. Continuous dependence also follows from the local result, we omit the details. This completes the proof of Theorem    \ref{thm-dkg reduced}.

          \bibliographystyle{amsplain} \bibliography{remark}
        \end{document}